\documentclass[a4paper,10pt]{article}

\usepackage{geometry}

\usepackage[latin1]{inputenc}  
\usepackage[english]{babel}
\usepackage[T1]{fontenc}                           
\usepackage{ams math}
\usepackage{amssymb}
\usepackage{mathrsfs}
\usepackage{amsfonts}
\usepackage{amsthm}
\usepackage{graphicx}
\usepackage{cases}
\usepackage{verbatim}

\newtheorem{theo}{Theorem}[section]
\newtheorem{prop}[theo]{Proposition}

\newtheorem{corol}[theo]{Corollary}

\theoremstyle{remark}
\newtheorem{rem}[theo]{Remark}

\newcommand{\ve}{\varepsilon}
\newcommand{\vp}{\varphi}

\newcommand{\Lg}{\mathcal{L}_g}

\newcommand{\RR}{\mathbb{R}}

\numberwithin{equation}{section}

\title{The Einstein-scalar field constraint system\\
in the positive case.}
\author{Bruno Premoselli}

\begin{document}

\maketitle

\begin{abstract}
We prove the existence of solutions to the conformal Einstein-scalar constraint system of equations for closed compact Riemannian manifolds in the positive case. Our results apply to the vacuum case with positive cosmological constant and to the massive Klein-Gordon setting. 
\end{abstract}

\section{Introduction}

\subsection{The Einstein constraint equations in a scalar-field theory and the conformal method.}

The constraint equations arise in general relativity. A space-time is a Lorentzian manifold $(\tilde{M}, \tilde{g})$ of dimension $n+1$ that solves the Einstein equations:
\begin{equation} \label{Ein}
\textrm{Ric}_{ij}(\tilde{g}) - \frac{1}{2}R(\tilde{g}) \tilde{g}_{ij} = 8 \pi T_{ij} 
\end{equation}
where $R(\tilde{g})$ and $\textrm{Ric}(\tilde{g})$ are respectively the scalar curvature and the Ricci curvature of $\tilde{g}$ and 
$T_{ij}$ is the stress-energy tensor. 
In a scalar-field theory the expression of $T$ involves a scalar field $\psi \in C^\infty(\tilde{M})$, a potential $V \in C^\infty(\RR)$, and the metric $\tilde{g}$, and is written as
\[ T_{ij} = \nabla_i \psi \nabla_j \psi - \left( \frac{1}{2}|\nabla \psi|_{\tilde{g}}^2 + V(\psi) \right)\tilde{g}_{ij} .\]
This setting covers several usual physical cases: for instance $V =0, \psi =0$ yields the vacuum constraint equations, $V \equiv \Lambda $ and $\psi = 0$ yields the vacuum case with a cosmological constant, and $V(\psi) = \frac{1}{2}m \psi^2$ yields the massive Klein-Gordon setting. An initial data set for the Einstein equations consists of 
$(M,g,K,\psi,\tilde{\pi})$, where $(M,g)$ is a $n$-dimensional closed compact Riemannian manifold ($n \ge 3$), $k$ is a symmetric $(2,0)$-tensor, and $\psi $ and $\tilde{\pi}$ are smooth functions in $M$. 
The Cauchy problem in general relativity deals with constructing space-time developments for given initial data sets $(M,g,K,\psi,\tilde{\pi})$. Such a development consists of a Lorentzian manifold $(M \times \mathbb{R},\tilde{g})$ and of a smooth function $\tilde{\psi}$ on $M \times \mathbb{R}$ such that $(M \times \mathbb{R},\tilde{g})$ solves the Einstein equations \eqref{Ein}, $g$ is the Riemannian metric induced by $\tilde{g}$ on $M$, $K$ is the second fundamental form of the embedding $M \subset M \times \RR$ and $\psi$ and $\tilde{\pi}$ are respectively the scalar field and its temporal derivative on $M$, that is $\tilde{\psi}_{|M} = \psi$ and $\partial_t \tilde{\psi}_{| M} = \tilde{\pi}$.
A necessary condition for the existence of a space-time development of an initial data set $(M,g,K,\psi,\tilde{\pi})$ is found applying the Gauss and Codazzi equations to \eqref{Ein} and yields the following well-known system of equations:
\begin{equation} \label{sys:cont}
\left \{ 
\begin{aligned}
& R(g) + \textrm{tr}_g K ^2  - \left| \left| K \right| \right|_{g}^2  = \tilde{\pi}^2 + |\nabla \psi|_{g}^2 + 2V(\psi)~,\\
&  \partial_i (\mathrm{tr}_{g}K) - K_{i,j}^j   = \tilde{\pi} \partial_i \psi~, \\
\end{aligned} \right. \end{equation}
where $R(g)$ is the scalar curvature of $(M,g)$. As shown first by Choquet-Bruhat \cite{ChoBru} for the vacuum case ($\psi = \tilde{\pi} = 0$) and later by Choquet-Bruhat-Isenberg-Pollack in the general case \cite{ChoIsPo2},  the system \eqref{sys:cont} is also a sufficient condition on $(M,g,K,\psi,\tilde{\pi})$ for the existence of a space-time development. 
This development is unique as shown by Choquet-Bruhat and Geroch \cite{ChoGe}. A survey reference on the subject is Chru\'sciel-Galloway-Pollack \cite{ChruGaPo}.

\medskip
Solving \eqref{sys:cont} provides admissible initial data for the Einstein equations. A method that turned out to be very effective to solve \eqref{sys:cont} is the conformal method, initiated by Lichnerowicz \cite{Lich}. It consists in turning \eqref{sys:cont} into a determined system by specifying some initial ``free'' data and to solve the system in the remaining ``determined'' initial data. The set of free data consists of $(\psi,\tau,\pi,U)$ where $\psi,\tau,\pi$ are smooth functions in $M$ and $U$ is a smooth symmetric traceless and divergence-free $(2,0)$-tensor in $M$. 
Given $( \psi, \tau, \pi, U)$ an initial free data set, the conformal method yields a system of two equations, often referred to as the conformal constraint system 
of equations, whose unknowns are a smooth positive function $\vp$ in $M$ and a smooth vector field $W$ in $M$. The conformal constraint system  is written as
\begin{numcases}{}
    \triangle_g \varphi + \mathcal{R}_\psi  \varphi  = \mathcal{B}_{\tau, \psi, V} \varphi^{2^*-1} + \frac{\mathcal{A}_{\pi, U}(W)}{ \varphi^{2^*+1}}~,  \label{EL} \\ 
  \triangle_{g, conf} W  = \frac{n-1}{n}\varphi^{2^*} \nabla\tau - \pi\nabla \psi~,    \label{M}
  \end{numcases} 
where we have let:
\begin{equation} \label{expressions}
\begin{aligned}
& \mathcal{R}_{\psi}  = c_n \left( R(g) - |\nabla \psi|_g^2 \right), \\
& \mathcal{B}_{\tau,\psi,V}  = c_n \left( 2 V(\psi) - \frac{n-1}{n} \tau^2 \right), \\
& \mathcal{A}_{\pi,U}(W)  = c_n \left(  |U + \mathcal{L}_g W |_g^2 + \pi^2 \right) ,\\
\end{aligned} \end{equation}
and $c_n = \frac{n-2}{4(n-1)}$. The notation $\mathcal{A}_{\pi,U}(W)$ emphasizes the dependency with respect to $W$, which is given by the second equation. 
In \eqref{expressions} we adopt similar notations to those in Choquet-Bruhat, Isenberg and Pollack \cite{ChoIsPo} except for the minus sign on $\mathcal{B}_{\tau, \psi,V}$. Also, in \eqref{EL}-\eqref{M}, 
$\triangle_g = - \textrm{div}_g \nabla$ is the Laplace-Beltrami operator, with nonnegative eigenvalues, $2^* = \frac{2n}{n-2}$ is the critical Sobolev exponent, $\triangle_{g, conf}W = \textrm{div}_g (\mathcal{L}_g W)$ and $\mathcal{L}_g W $ is the symmetric trace-free part of $\nabla W$:
\begin{equation} \label{conflie}
  \mathcal{L}_gW_{ij} = W_{i,j} + W_{j,i} - \frac{2}{n} \mathrm{div}_g W g_{ij}. 
\end{equation}
The first equation is referred to as the Einstein-Lichnerowicz equation while the second one is referred to as the momentum constraint. Smooth vector fields in the kernel of $\mathcal{L}_g$ are called conformal Killing vector fields.  
Since $M$ is compact without boundary, the integration by parts formula gives, for any smooth vector field $W$:
\[ \triangle_{g,conf} W = 0 \iff \mathcal{L}_g W = 0. \] 
Given an initial data set $(\psi, \tau, \pi, U)$, if $(\vp, W)$ solves \eqref{EL}-\eqref{M} then
\begin{equation} \label{initialdata}
\left( M, \vp^{\frac{4}{n-2}}g, \frac{\tau}{n} \vp^{\frac{4}{n-2}} g + \vp^{-2}(U + \mathcal{L}_g W), \psi ,\vp^{-\frac{2n}{n-2}}\pi \right)
\end{equation}
is a solution of \eqref{sys:cont}. In this case $\tau$ is the mean curvature of $(M,  \vp^{\frac{4}{n-2}}g)$ embedded in its space-time development, $\psi$ is the scalar-field restricted to $M$ and, up to a conformal factor, $\pi$ is the time derivative of the scalar-field in $M$. We refer to Choquet-Bruhat, Isenberg and Pollack \cite{ChoIsPo} and Bartnik-Isenberg \cite{BarIse} for more developments on the conformal method.  Solving the constraint system in usual cases such as the massive Klein-Gordon setting or the positive cosmological constant case amounts to solve \eqref{EL}-\eqref{M} for a good choice of the initial data.

\subsection{Statement of the results}

In this paper we focus on the conformal constraint system \eqref{EL}-\eqref{M} in the case of a non-negative potential $V$.
If $h$ is a smooth function in $M$, $\triangle_g + h$ is said to be coercive if there exists a positive constant $C$ such that for any $u \in H^1(M)$,
\[  \int_M \left( |\nabla u|_g^2 + hu^2 \right) dv_g \geqslant C ||u||_{H^1(M)}^2 \]
or, equivalently, if
\begin{equation} \label{Nh}
\Vert u \Vert_{H^1_h} =  \left( \int_M \left( |\nabla u|_g^2 + hu^2 \right) dv_g \right)^{\frac{1}{2}}
\end{equation}
is an equivalent norm on $H^1(M)$. In this case, following Hebey, Pacard and Pollack \cite{HePaPo}, we define the constant $S_h$ to be the smallest positive constant satisfying 
that for all $u \in H^1(M)$:
\begin{equation} \label{Sh}
\Vert u \Vert_{L^{2^*}} \le S_{h}^{\frac{1}{2^*}} \Vert u \Vert_{H^1_h}.
\end{equation}
Our main result states the existence of a solution $(\vp, W)$ of \eqref{EL}-\eqref{M} in the positive case under suitable smallness assumptions on the free data. It is stated as follows.
\begin{theo} \label{Th}
Let $(M,g)$ be a closed compact Riemannian manifold of dimension $n \ge 3$ of positive Yamabe type such that $g$ has no conformal Killing vector fields. Let $V$ be a smooth nonnegative function on $\RR$, $V \not \equiv 0$,  
and let  $\psi$ be a smooth function in $M$ such that the operator $\triangle_g + \mathcal{R}_\psi$ is coercive.
There exists a positive constant $\ve \left( n,g, V, \psi  \right)$ depending only on $n,g,  \sup_{x \in M} V(\psi(x))$ and $ S_{\mathcal{R}_\psi}$ as in \eqref{Sh} 
such that if the remaining part of the initial data $(\tau, \pi, U)$ satisfies
\begin{equation} \label{signeB}
\frac{n-1}{n} \tau^2(x) \le 2  V \left( \psi(x) \right)  \textrm{ for all } x \in M ,
\end{equation} 
the equality being strict somewhere in $M$,  $\Vert \pi \Vert_\infty + \Vert U \Vert_\infty > 0$ and
\begin{equation} \label{controle}
\Vert \nabla \tau \Vert_\infty + \Vert \pi \Vert_\infty + \Vert U \Vert_\infty \le \ve ( n,g, V, \psi )~,
\end{equation}
then the conformal constraint system \eqref{EL}-\eqref{M} has a solution $(\vp, W)$.
\end{theo}
\begin{rem}
When the initial data satisfies condition \eqref{signeB}, by the notations in \eqref{expressions}, $\mathcal{B}_{\tau, \psi, V}$ is non-negative in $M$ and positive somewhere. As we shall see in Section \ref{nonex}, when $\mathcal{B}_{\tau, \psi, V}$ is non-negative, then being of positive Yamabe type is a necessary condition.
\end{rem}
\medskip
To our knowledge, Theorem \ref{Th} is the first result in the non-CMC setting ($ \nabla \tau \not\equiv 0$) 
 when $\mathcal{B}_{\tau, \psi, V} \ge  0$. With this result we get the existence of admissible initial data in important cases such as the massive Klein-Gordon setting with nonzero potential or the positive cosmological constant case. Due to its importance we state the latter separately:
\begin{corol} \label{cosmo}
Let $\Lambda$ be a positive constant. Then the vacuum conformal constraint system of equations with positive cosmological constant  $\Lambda$, namely
\begin{equation*}
\left \{
\begin{aligned}
& \frac{4(n-1)}{n-2} \triangle_g \varphi + R(g)  \varphi  = \left( 2 \Lambda  - \frac{n-1}{n} \tau^2 \right) \varphi^{2^*-1} +  \left| U + \mathcal{L}_g W \right|_g^2 \varphi^{-2^*-1}~,  \\
& \triangle_{g, conf}W  = \frac{n-1}{n}\varphi^{2^*} \nabla\tau~, \\
\end{aligned}
\right.
\end{equation*}
has a solution $(\vp,W)$ provided $U \not \equiv 0$ and
\begin{equation} \label{controle2}
\Vert \tau \Vert_{C^1} + \Vert U \Vert_\infty \le C(n,g,\Lambda)~,
\end{equation}
where $C(n,g,\Lambda)$ is some constant depending only on $n,g$ and $\Lambda$.
\end{corol}

In this scalar-field setting the smallness assumptions \eqref{signeB} and \eqref{controle} only involve the scalar field $\psi$ and the potential $V$, which is itself related to $\psi$ by some wave equation that expresses the conservation of energy, see Wald \cite{Wa}. This emphasizes the influence of $\psi$ which appears to be the important parameter to consider. 

\medskip
There are several interesting results on systems like  \eqref{EL}-\eqref{M}. 
They can be roughly classified according to two criteria: (i) the CMC (constant mean curvature) versus the non-CMC case, 
and, if we forget about the fact that $\mathcal{B}_{\tau, \psi, V}$ may change sign, (ii) the positive case, where $\mathcal{B}_{\tau, \psi, V} > 0$, versus  
the nonpositive case, where $\mathcal{B}_{\tau, \psi, V} \le 0$. 
In the CMC setting ($\nabla \tau = 0$) the system \eqref{EL}-\eqref{M} is semi-decoupled. Equation \eqref{M} is solvable, either assuming that there are no conformal Killing fields on $M$ or assuming that $\pi\nabla \psi$ is orthogonal to such fields (which generically do not exist, see Beig-Chru\'sciel-Schoen \cite{BeChSc}). 
Its solution appears as a coefficient in \eqref{EL}. In the CMC-case when $\mathcal{B}_{\tau,\psi,V} \le 0$, for instance in the vacuum case, the system is fully understood (see Isenberg \cite{Ise} or Choquet-Bruhat, Isenberg and Pollack \cite{ChoIsPo}). Partial results exist in the $\max_M \mathcal{B}_{\tau,\psi,V}  > 0 $ case, and we refer to Hebey, Pacard and Pollack \cite{HePaPo}, and Ng\^o and Xu \cite{NgoXu}. In the non-CMC case, results were available only when $\mathcal{B}_{\tau,\psi,V} \le 0$ and assuming smallness assumptions on the initial data. For near-CMC results see Allen-Clausen-Isenberg \cite{ACI} or Dahl-Gicquaud-Humbert \cite{DaGiHu}. Results when $U$ is small can be found in Holst-Nagy-Tsogtgerel \cite{HoNaTso} or Maxwell \cite{Maxwell}. A few non-existence results exist for near-CMC initial data: see Isenberg-\`O Murchadha \cite{IsMu} or again Dahl-Gicquaud-Humbert \cite{DaGiHu}. A condition like our condition \eqref{controle2} is both a near-CMC assumption and a control 
on $\Vert U \Vert_\infty$, and Corollary \ref{cosmo} can be thought as a generalization of the available existence results for the vacuum conformal constraint system 
of equations to the more involved case where $\mathcal{B}_{\tau, \psi, V} > 0$.

\medskip
The paper is organized as follows. In section \ref{nonex} we comment on Theorem \ref{Th}. We prove necessary conditions for the existence of solutions of \eqref{EL}-\eqref{M} and show that the need of a control on the initial data is natural. Section \ref{solumini} is devoted to the proof of the existence of a smallest solution for equation \eqref{EL}. Theorem \ref{Th} is proved in section \ref{pingpong} using a fixed-point argument.

\paragraph{Acknowledgements.}
The author wishes to thank warmly Olivier Druet and Emmanuel Hebey for many stimulating discussions and useful comments on this paper.

\section{Necessary conditions and non-existence results.}  \label{nonex}

We discuss the assumptions of Theorem \ref{Th}. Throughout this section we assume that $V$ is a smooth nonnegative function in $\RR$, not everywhere zero.  We let
\begin{equation} \label{infcoercif}
 \mu_g = \inf_{\vp \in H^1(M), \Vert \vp \Vert_{2} = 1} \int_M \left( |\nabla \vp|^2 + c_n R(g) \vp^2 \right) dv_g 
 \end{equation}
be the coercivity constant of $\triangle_g + c_n R(g)$, where $c_n$ is as in \eqref{expressions}, and
\begin{equation} \label{infcoercifpsi}
 \mu_{g,\psi} = \inf_{\vp \in H^1(M), \Vert \vp \Vert_{2} = 1} \int_M \left( |\nabla \vp|^2 + \mathcal{R}_\psi \vp^2 \right) dv_g 
 \end{equation}
be the analogue for $\triangle_g + \mathcal{R}_\psi$. The first result we prove shows that the coercivity of $\triangle_g + \mathcal{R}_\psi$ is a necessary condition to the existence of admissible initial data when $\mathcal{B}_{\tau,\psi,V}$ is non-negative. This shows, in some sense, the optimality of the assumptions required on $\psi$ in Theorem \ref{Th}. As a by-product we obtain an integral control on $|\nabla \psi|$, which in turns implies a strong geometric condition on the underlying manifold which shows a radically different behavior than in the vacuum case.

\begin{prop}
Let $(\psi, \tau, \pi, U)$ be an initial data set such that $ \mathcal{B}_{\tau,\psi,V} \ge 0 $ on $M$, where $\mathcal{B}_{\tau,\psi,V} $ is as in \eqref{expressions}, but $\mathcal{B}_{\tau,\psi,V} $ is not everywhere zero. If a solution $(\vp, W)$ of \eqref{EL}-\eqref{M} exists, then $\triangle_g + \mathcal{R}_\psi$ is coercive and there holds that
\begin{equation} \label{bornepsi}
 \int_M |\nabla \psi|^2 dv_g <  \int_M  R(g) dv_g.
 \end{equation}
In particular, $(M,g)$ is of positive Yamabe type and we get both that $ \mu_{g,\psi}  > 0$ and $\mu_g > 0$.
\end{prop}

\begin{proof}
Using standard variational techniques and elliptic theory 
we easily obtain that there exists a smooth positive function $u_{g,\psi}$ with $\Vert u_{g,\psi} \Vert_2 = 1$ such that 
\begin{equation} \label{equa2psi}
\triangle_g u_{g,\psi} + \mathcal{R}_\psi u_{g,\psi} = \mu_{g,\psi} u_{g,\psi}
 \end{equation}
where $\mu_{g,\psi}$ is as in \eqref{infcoercifpsi}. Since $(\vp,W)$ solves \eqref{EL}-\eqref{M} and $\mathcal{B}_{\tau,\psi,V}$ is non-negative, $\mathcal{B}_{\tau,\psi,V} \not \equiv 0$, integrating \eqref{EL} against $u_{g,\psi}$ and using \eqref{equa2psi} shows that $\mu_{g,\psi} > 0$. It is well-known (see again \cite{HeDruRob}) that this implies the coercivity of $\triangle_g + \mathcal{R}_\psi$  which implies in particular that $\int_M \mathcal{R}_\psi dv_g > 0$ and yields \eqref{bornepsi}. 
Assuming by contradiction that the Yamabe type of $(M,g)$ is nonpositive, we get that there exists $\tilde g \in [g]$, where $[g]$ is the conformal class of $g$, with $R({\tilde g}) \le 0$ 
in $M$. Writing that $g = v^{4/(n-2)}\tilde g$ for $v > 0$, there holds that
$$\Delta_{\tilde g}v + c_n R({\tilde g})v = c_n R(g) v^{2^\star-1}$$
Dividing the equation by $v^{2^\star-1}$ and integrating the contradiction follows from \eqref{bornepsi}. 
\end{proof}

Now we discuss a non-existence result which shows the necessity of a control on $\pi$ depending on $\mathcal{B}_{\tau, \psi,V}$. More precisely, the following 
result by Hebey, Pacard and Pollack \cite{HePaPo} holds true.
\begin{prop} \label{nonexistence}
Let $(\tau, \psi)$ be smooth functions with $\mathcal{B}_{\tau, \psi, V} >0$ in $M$, where $\mathcal{B}_{\tau, \psi, V}$ is as in \eqref{expressions}. If $\pi$ is a smooth function in $M$ satisfying
\begin{equation*}
\int_M  \pi^{\frac{n+2}{2n}} dv_g 
   > \left( \frac{(n-1)^{n-1}c_n^{n-1}}{n^n} \right)^{\frac{n+2}{4n}} 
 \frac{\int_M |R(g)|^{\frac{n+2}{4}} dv_g}{\left( \min_{x \in M } \mathcal{B}_{\tau, \psi,V}(x) \right)^{\frac{(n-1)(n+2)}{4n}}} \hskip.1cm ,
  \end{equation*}
then the system \eqref{EL}-\eqref{M} admits no solutions with $(\psi, \tau, \pi, U)$ as initial data set for any smooth traceless and divergence-free $(2,0)$-tensor $U$.
\end{prop}

\begin{proof}
Let $W$ be a smooth vector field in $M$. Following Hebey-Pacard-Pollack \cite{HePaPo} we get that \eqref{EL} has no solutions if
\begin{equation} \begin{aligned} \label{hepa}
\int_M \mathcal{A}_{\pi,U}(W)^{\frac{n+2}{4n}}  > \left( \frac{(n-1)^{n-1}}{n^n}\right)^{\frac{n+2}{4n}}   \left( \min_M \mathcal{B}_{\tau,\psi,V} \right)^{-\frac{(n-1)(n+2)}{4n}} \int_M ( \mathcal{R}_\psi^+)^{\frac{n+2}{4}}dv_g\hskip.1cm , \\
\end{aligned} 
\end{equation}
where we used the notations of \eqref{expressions} and where for any function $f$ we write $f^+ = \max(f,0)$. 
We prove \eqref{hepa} by contradiction. We assume that \eqref{EL} has a smooth positive  solution $\vp$. First, we integrate \eqref{EL} to get
\begin{equation} \label{intEL}
 \int_M \mathcal{B}_{\tau, \psi,V} \vp^{2^*-1} dv(g) + \int_M \mathcal{A}_{\pi,U}(W) \vp^{-2^*-1} dv(g) = \int_M \mathcal{R}_\psi \vp dv(g)  
 \end{equation}
and then we apply a H\"{o}lder inequality with parameters $ \frac{n+2}{4}$ and $\frac{n+2}{n-2}$ to the right-hand side of \eqref{intEL}. This yields
\[ \begin{aligned}
\int_M \mathcal{R}_\psi \vp dv(g) \le 
  \left( \int_M (\mathcal{R}_\psi^+)^{\frac{n+2}{4}} \mathcal{B}_{\tau,\psi,V}^{-\frac{n-2}{4}} dv(g) \right)^{\frac{4}{n+2}} \left( \int_M \mathcal{B}_{\tau,\psi,V} \vp^{ 2^*-1} dv(g) \right)^{\frac{n-2}{n+2}} \\
 \end{aligned} .\]
Independently, a H\"{o}lder inequality with parameters $\frac{4n}{n+2}$ and $\frac{4n}{3n-2}$ yields
\[ \begin{aligned}
\int_M \mathcal{A}_{\pi,U}(W)^{\frac{n+2}{4n}} \mathcal{B}_{\tau,\psi,V}^{\frac{3n-2}{4n}} dv(g)
 & \le  \left( \int_M \mathcal{A}_{\pi,U}(W) \vp^{-2^*-1} dv(g) \right)^{\frac{n+2}{4n}} \left( \int_M \mathcal{B}_{\tau,\psi,V} \vp^{2^*-1} dv(g) \right)^{\frac{3n-2}{4n}}
 \hskip.1cm , \\
\end{aligned} \]
so that, letting
$X = \left( \int_M \mathcal{B}_{\tau, \psi,V} \vp^{2^*-1} dv(g) \right)^{\frac{4}{n+2}}$, 
equation \eqref{intEL} gives
\begin{equation}\label{AddedEqtNonEx}
\left( \int_M (\mathcal{R}_\psi^+)^{\frac{n+2}{4}} \mathcal{B}_{\tau, \psi,V}^{-\frac{n-2}{4}} dv(g) \right)^{\frac{4}{n+2}} \geqslant X + \left( \int_M \mathcal{A}_{\pi,U}(W)^{\frac{n+2}{4n}} \mathcal{B}_{\tau, \psi,V}^{\frac{3n-2}{4n}} dv(g) \right)^{\frac{4n}{n+2}}X^{1-n}~.
\end{equation}
Let $K_X$ be the right hand side in \eqref{AddedEqtNonEx}. The minimum value of $K_X$ as a function of $X$ is:
\begin{equation}\label{AddedEqt2NonEx}
\min_{X > 0}K_X = \frac{n}{(n-1)^{\frac{n-1}{n}}} \left(  \int_M \mathcal{A}_{\pi,U}(W)^{\frac{n+2}{4}} \mathcal{B}_{\tau, \psi, V}^{\frac{3n-2}{4n}} dv_g \right)^{\frac{4}{n+2}} 
\end{equation}
and the non-existence condition \eqref{hepa} easily follows from \eqref{AddedEqtNonEx} and \eqref{AddedEqt2NonEx} by contradiction.
Then Proposition \ref{nonexistence} follows from \eqref{hepa} since $\mathcal{A}_{\pi,U}(W) \ge c_n \pi^2$ and  
$\mathcal{R}_\psi^+ \le c_n|R(g)|$ by \eqref{expressions}.
\end{proof}

\section{A minimal solution of the Einstein-Lichnerowicz equation} \label{solumini}

In the constant mean curvature setting the constraint system is completely decoupled and it reduces to the Einstein-Lichnerowicz equation \eqref{EL}. We now investigate  \eqref{EL} independently and for the sake of clarity consider the following equation: 
\begin{equation} \label{ELa} \tag{$EL_a$}
\triangle_g u + h u  = f u^{2^*-1} + \frac{a}{u^{2^*+1}}~,
\end{equation}
where $h, f$ and $a$ are smooth functions on $M$. 
In the following we assume that $\triangle_g + h$ is coercive, $\max_M f >0$ and $a$ is nonnegative and nonzero. Using repeatedly the sub and super solution method we prove that each time equation \eqref{ELa} has a smooth positive solution then it has a smallest solution for the $L^\infty$-norm:

\begin{prop} \label{solmin}
Let $a \ge 0$ be a nonzero smooth function in $M$. Assume that $\triangle_g + h$ is coercive and $\max_M f >0$. If \eqref{ELa} has a smooth positive solution 
then there exists a smooth positive function $\vp(a)$ solving \eqref{ELa} such that for any other solution $\vp$, with $\vp \not\equiv \vp(a)$, there holds that $\vp(a) < \vp$ in $M$. 
Moreover, $\vp(a)$ is stable in the sense that for any $\theta \in H^1(M)$,
\[ \int_M\left( |\nabla \theta|^2 + \left [h - (2^*-1)f \vp(a)^{2^*-2} + (2^*+1) \frac{a}{ \vp(a)^{2^*+2}} \right] \theta^2 \right) dv_g \ge 0~, \]
and $a \to \vp(a)$ is also nondecreasing with respect to $a$ in the sense that  if $a_1 \le a_2$ in $M$, provided that $\vp(a_1)$ and $\vp(a_2)$ exist, then there holds that $\vp(a_1) \le \vp(a_2)$.
\end{prop}

\begin{rem}
We prove Proposition \ref{solmin} assuming that $a$ is smooth but the result still holds if $a$ is only continuous. In this case the minimal solution we obtain belongs to $C^{1,\alpha}(M)$ for any $0 < \alpha < 1$.
\end{rem}

\begin{proof}
Let $a \ge 0$ be a nonzero smooth function such that \eqref{ELa} has a solution. We start proving that there exists a positive number that bounds from below all the solutions of \eqref{ELa}. First, let us notice that there always exist sub-solutions of \eqref{ELa} as small as we want. Indeed, for any $\delta \ge 0$ we let $u_\delta$ be the unique solution of
\begin{equation} \label{udelta}
 \triangle_g u_\delta + h u_\delta = a - \delta f^-  -\delta
 \end{equation}
where $f^- = - \min(f,0)$. Since $a$ is nonnegative and nonzero, the maximum principle shows that $u_0 > 0$ in $M$. Since
\[ \left( \triangle_g + h \right) (u_0 - u_\delta) = \delta f^-  + \delta~,\]
we obtain by standard elliptic theory that $\Vert u_\delta - u_0 \Vert_\infty \to 0$ as $\delta$ tends to $0$. Let $\delta_0 > 0 $ be small enough in order to have $u_{\delta_0} >0$. Then for $\ve$ small enough, 
\begin{equation} \label{soussol}
v_\ve = \ve u_{\delta_0}
\end{equation}
is a strict sub-solution of \eqref{ELa} since, by \eqref{udelta},
\[ \triangle_g v_\ve + h v_\ve = \ve a - \ve \delta_0 f^-  - \ve \delta_0< \frac{a}{v_\ve^{2^*+1}} + f v_\ve^{2^*-1}. \]
Now we claim that there exists some $\ve_0 > 0$ such that for any positive solution $\vp$ of \eqref{ELa} there holds 
\begin{equation} \label{minunif}
\vp > v_{\ve_0}
\end{equation} 
in $M$, where $v_{\ve_0}$ is as in \eqref{soussol}. We prove the claim by contradiction and assume that 
there exists $\vp_\ve$ solution of \eqref{ELa}, and $x_\ve \in M$, such that $\vp(x_\ve) \le v_\ve(x_\ve)$ for all 
$\ve > 0$. Then, for some $\tilde\ve \in (0,\ve)$, and some $\tilde x_\ve \in M$, 
\[ 1 = \inf_M \frac{\vp}{v_{\tilde\ve}} = \frac{\vp(\tilde x_\ve)}{v_{\tilde\ve}(\tilde x_\ve)}~.\]
In particular, we obtain that
\[v_{\tilde\ve}(\tilde x_\ve) = \vp(\tilde x_\ve)~~\hbox{and}~\triangle_g \vp (\tilde x_\ve) \le \triangle_g v_{\tilde\ve} (\tilde x_\ve) \]
which is impossible since $v_{\tilde\ve}$ is a strict subsolution of \eqref{ELa}.

\medskip
We prove now the existence of a minimal solution of \eqref{ELa}. We follow here the arguments in Sattinger \cite{Sattinger}. For $x \in M$ and $u >0$ we let 
\[F(x,u) = f(x) u(x)^{2^*-1} + \frac{a(x)}{u(x)^{2^*+1}} - h(x) u(x).\] 
Let $\psi$ be a solution of \eqref{ELa} and let $w$ be a strict subsolution of \eqref{ELa} which is less than any positive solution of \eqref{ELa}. We proved the existence of such a $w$ 
in \eqref{minunif}. Also we let $K > 0$ be large enough such that for any $x\in M$, and any $\min_M w  \le u \le \max_M \psi$,
\begin{equation} \label{propK}
\begin{aligned}
 F(x,u) + Ku \ge 0 \textrm{ and } \frac{\partial F}{\partial u}(x,u) + K \ge 0.
 \end{aligned}
 \end{equation} 
 For any smooth positive function $u$, we define $Tu$ as the unique solution of
\begin{equation} \label{defiT}
\triangle_g Tu + K Tu = F(\cdot,u) + Ku .
\end{equation}
As a first remark, for any two positive functions $u$ and $v$ in the range 
\[ \min_M w \le u,v \le \max_M \psi \]
we have: 
\[ \big( \triangle_g + K \big)(Tu - Tv)(x) = F(x,u) - F(x,v) + K \big( u(x) - v(x) \big).\]
Then, by the strong maximum principle, we obtain that
\begin{equation} \label{croissance}
Tu < Tv \textrm{ as long as } u \le v ~\textrm{and } u \not \equiv v.
\end{equation} 
The iterative sub and super solution method applied in the range $w \le \vp \le \psi$ and starting from the strict sub-solution $w$ provides a sequence $v_n = T^n w$ which is non decreasing by the maximum principle and converges to a fixed-point of $T$, that is to say a solution of \eqref{ELa} (see \cite{Sattinger} for more details). We shall call this solution $\vp(a)$:
\begin{equation} \label{defsol}
\vp(a) = \lim_{n \to \infty} T^n w.
\end{equation}
By standard elliptic theory, $\vp(a)$ is smooth. Note in passing that all the above arguments still work if we only assume that $a$ is continuous, but in this case $\vp(a)$ constructed as in \eqref{defsol} will only be of class $C^{1,\alpha}$ for any $0 < \alpha < 1$.

\medskip
Now we show that $\vp(a)$ does not depend on $\psi$ and on $w$. First, $\vp(a)$ as in \eqref{defsol} does not depend on $\psi$. We let $\psi_1$ and $\psi_2$ be two solutions of \eqref{ELa}. We let $K_i$, $i=1,2$ be positive constants satisfying \eqref{propK} in $[ \min_M w ; \max_M \psi_i]$, $T_i$ be the operator defined as in \eqref{defiT} and $\vp_i$ the associated solution as in \eqref{defsol}. Since $(T_1^n w)$ is non decreasing there holds $\vp_1 \ge w$. If we assume for instance that $\max_M \psi_1 \le \max_M \psi_2$ then $\vp_1 \in [ \min_M w ; \max_M \psi_2]$ and thus, by \eqref{defsol} and the maximum principle there holds $\vp_2 \le \vp_1$ 
since $T_2(\vp_1) = \vp_1$. But then $\vp_2$ is a solution of \eqref{ELa} with $\min_M w \le \vp_2 \le \max_M \psi_1$ and thus, once again by the maximum principle, $\vp_1 \le \vp_2$. This proves that $\vp(a)$ does not depend on $\psi$. 
Now we prove that $\vp(a)$ does not depend on the strict subsolution $w$, provided that $w$ is less than any positive solution of \eqref{ELa}. Indeed, for any $\psi$ solution of \eqref{ELa}, if $w_1$ and $w_2$ are 
two such subsolutions, and $\vp_1$ and $\vp_2$ are the associated solutions as in \eqref{defsol}, then there holds $w_1 \le \vp_2$ and $w_2 \le \vp_1$. We conclude once again with the maximum principle that 
$\vp_1 \le \vp_2$ and $\vp_2 \le \vp_1$. 

\medskip
By the definition of $\vp(a)$ in \eqref{defsol}, and what we just proved, 
for any $\psi$ solution of \eqref{ELa} there holds that $w < \vp(a) \le \psi$ where $w$ is a subsolution that is less than any solution of \eqref{ELa}. With \eqref{croissance} we obtain the desired property:
\begin{equation} \label{minimal}
\vp(a) < \psi \textrm{ or } \vp(a) \equiv \psi.
 \end{equation}
The stability of $\vp(a)$ is a consequence of the minimality of $\vp(a)$. We denote by $\lambda_0$ the first eigenvalue of the linearized operator of equation \eqref{ELa} at $\vp(a)$. The stability of $\vp(a)$ as stated in Proposition \ref{solmin} amounts to say that $\lambda_0 \ge 0$. Assume by contradiction that $\lambda_0 < 0$ and denote by  $\psi_0$ the associated positive eigenvector. Let $w$ be a subsolution that is less than any solution of \eqref{ELa}. Let $\vp_\delta = \vp(a) - \delta \psi_0$ for any positive $\delta$. For $\delta > 0$ small enough one has
\[ w <  \vp_\delta < \vp(a) \]
and a straightforward calculation shows that
\[ \triangle_g  \vp_\delta +h \vp_\delta - f \vp_\delta^{2^*-1} - \frac{a}{\vp_\delta^{2^*+1}} =  - \delta \lambda_0 \psi_0 + o(\delta) > 0\] 
so that $\vp_\delta$ is a strict supersolution of \eqref{ELa} satisfying $w < \vp_\delta < \vp(a)$ for $\delta$ small enough. By the iterative sub and super solution method we then get a solution
$\psi$ of $(EL_a)$ such that $w < \psi < \varphi_\delta$, and this is in contradiction with \eqref{minimal}.

\medskip
Finally, if $a_1 \le a_2$ are nonnegative nonzero functions on $M$, $\vp(a_2)$ is a super solution of equation \eqref{ELa} with $a = a_1$. By the minimality of $\vp(a_1)$ we then have $\vp(a_1) \le \vp(a_2)$.
 \end{proof}

\section{The fixed-point method: proof of Theorem \ref{Th}} \label{pingpong} 

In this section we prove Theorem \ref{Th} using the results obtained in the previous section. Before we start let us recall some basic elliptic properties of the operator $\triangle_{g,conf}$ that can be found for instance in Isenberg and \`{O} Murchadha \cite{IsMu}. The following 
proposition can be found in  \cite{IsMu}.

\begin{prop} \label{estimee}
Let $(M,g)$ be a closed compact Riemannian manifold of dimension $ n\ge3$ such that $g$ has no conformal Killing fields. Let $X$ be a smooth vector field in $M$. Then there exists a unique solution
$W$  of
\[ \triangle_{g,conf} W = X .\]
Also, there exists a constant $C_0>0$ that depends only on $n$ and $g$ such that
\[ \Vert W \Vert_{C^{1,\alpha}} \le C_0 \Vert X \Vert_\infty\]
for some positive $\alpha$. As a straightforward consequence, there exists a constant $C_1$ still depending only on $n$ and $g$ such that 
$\Vert \mathcal{L}_g W \Vert_\infty \le C_1 \Vert X \Vert_\infty$.
\end{prop}

The proof of Theorem \ref{Th} is obtained through a standard fixed-point argument. We develop the proof in what follows.

\paragraph{Obtaining a first estimate.}

Let $(M,g)$ be a closed compact Riemannian manifold of dimension $n \ge 3$ of positive Yamabe type such that $g$ has no conformal Killing vector fields. Let $V$ be a smooth nonnegative function in $\RR$, non everywhere zero and $\psi$ be a smooth function in $M$ such that $\triangle_g + \mathcal{R}_\psi$ is coercive. Also let $\pi$ and $U$ be such that $(\pi,U)$ is not everywhere zero in $M$, and let $\tau$ be a smooth function in $M$ such that 
\[\frac{n-1}{n} \tau^2(x) \le 2  V \left( \psi(x) \right)  \textrm{ for all } x \in M ,\] 
the equality being strict somewhere. This means, with the notations of \eqref{expressions}, that $\mathcal{B}_{\tau, \psi,V}$ is non-negative in $M$ and positive somewhere. In order to define the mapping to which 
we are going to apply Schauder's fixed-point Theorem, we need to get some important preliminary estimate based on  \eqref{controle}.
Let $W$ be a smooth vector field in $M$. We let
\begin{equation} \label{C(n,M)}
C(n,g,V, \psi) =  
C(n) V_g^{-1}  \left( 2 c_n S_{\mathcal{R}_\psi} \max_{x\in M} V(\psi(x)) \right)^{1-n} \left( \int_M \mathcal{R}_\psi dv_g \right)^{- \frac{2^*}{2}}
\end{equation}
where $V_g$ is the volume of $(M,g)$, $\mathcal{R}_\psi$ is as in \eqref{expressions}, $S_{\mathcal{R}_\psi}$ is as in \eqref{Sh} and $C(n) = \frac{1}{n-2} \left( 2(n-1)\right)^{-\frac{2^*}{2}}$. 
By \eqref{Sh} the constant $S_{\mathcal{R}_\psi}$ only depends on $n,g$ and on the coercivity constant of $\triangle_g + \mathcal{R}_\psi$, hence on $n,g$ and $ \nabla\psi $. We consider the equation
\begin{equation} \label{vp0}
\triangle_g \vp + \mathcal{R}_\psi \vp = \mathcal{B}_{\tau, \psi,V} \vp^{2^*-1} + C(n,g ,V,\psi)\vp^{-2^*-1}
\end{equation}
By the result in Hebey-Pacard-Pollack \cite{HePaPo}, and since by \eqref{expressions} we have that 
$2c_nV \ge \mathcal{B}_{\tau, \psi,V}$, 
\eqref{vp0} has a smooth positive solution. 
Since we assumed 
 $\triangle_g + \mathcal{R}_\psi$ coercive, using Proposition \ref{solmin} we can let $\vp_m $ be the minimal solution of \eqref{vp0} 
and let
\begin{equation} \label{defN1}
N_m =  \Vert \vp_m \Vert_\infty.
\end{equation} 
Let $L^\infty_+(M)$ be the set of non negative bounded functions in $M$. Regarding the vector equation \eqref{M}, since we have assumed that $g$ has no Killing vector 
fields, for any $\eta \in L^\infty_+(M)$ we can use Proposition \ref{estimee} to let $W(\eta)$ be the unique vector field solution of
\begin{equation} \label{Weta}
 \triangle_{g,conf} W(\eta) = \frac{n-1}{n}\eta^{2^*} \nabla\tau - \pi\nabla \psi . 
 \end{equation}
Proposition \ref{estimee} shows that 
\begin{equation} \label{estimeeW}
\Vert \mathcal{L}_g W(\eta) \Vert_\infty \le C_1 \left( \Vert \nabla \tau \Vert_\infty  \Vert \eta \Vert_{\infty}^{2^*}  + \Vert \pi \Vert_\infty \Vert \nabla \psi \Vert_\infty \right).
\end{equation}
By \eqref{C(n,M)}, \eqref{vp0} and \eqref{defN1}, $N_m$ depends only on $n,g, S_{\mathcal{R}_\psi}$ and $\max_M V(\psi)$. Using \eqref{expressions} and 
\eqref{estimeeW} it is easily seen that there exists a positive constant $\ve(n,g,V, \psi)$ depending only on $n,g, S_{\mathcal{R}_\psi}$, and $\max_M V(\psi)$, such that whenever
\begin{equation} \label{controlebis}
\Vert \nabla \tau \Vert_\infty + \Vert \pi \Vert_\infty +  \Vert U \Vert_\infty \le \ve(n,g,V,\psi)~,
\end{equation}
then 
\begin{equation} \label{estimeeLgW}
\left| \left|   \mathcal{A}_{\pi, U}(W(\eta)) \right| \right|_\infty 
< C(n,g,V,\psi)
\end{equation}
for any $\Vert \eta \Vert_\infty \le N_m$, where $\mathcal{A}_{\pi, U}(W(\eta))$ is as in \eqref{expressions}. After a straightforward computation, using \eqref{expressions} and \eqref{estimeeW}, one sees that, 
in order to obtain \eqref{estimeeLgW}, it is enough to assume that 
\begin{equation} \label{epsng}
\ve(n,g,V, \psi)^2 <  \frac{C(n,g,V,\psi)}{\left(3 + 4 C_1^2 (N_m^{2 \cdot 2^*} + \| \nabla \psi \|_\infty^2)\right)c_n}
\end{equation}
where $c_n$ is as in \eqref{expressions}, $C_1$ is obtained in Proposition \ref{estimee}, $C(n,g,V, \psi)$ is as in \eqref{C(n,M)} and $N_m$ is as in \eqref{defN1}. 
Now we construct the map $\mathcal{T}$ to which we are going to apply Schauder's fixed point theorem.

\paragraph{Definition of the mapping $\mathcal{T}$.} 

From now on, we will always assume that \eqref{controlebis} is satisfied, so that \eqref{estimeeLgW} holds true. For any positive $N$, we define
\begin{equation} \label{bn}
B_{N} = \{ \eta \in L_+^\infty(M), \Vert \eta \Vert_{\infty} \le N \}.
\end{equation} 
An easy claim is that 
for any vector field $W$ of class $C^1$ in $M$, $\mathcal{A}_{\pi,U}(W)$ as in \eqref{expressions} is continuous, non-negative and positive somewhere in $M$.
Obviously, $\mathcal{A}_{\pi,U}(W)$ is continuous and non-negative, and we just need to prove that it is positive somewhere. 
By \eqref{expressions} this is automatically true if $\pi$ is not everywhere zero. In case $\pi \equiv 0$ there might be that 
there exists a $C^1$ vector field in $M$ such that $U + \Lg W = 0$ everywhere. Taking the divergence of this equality in the weak sense yields, since $U$ is divergence-free:
\[ \triangle_{g,conf} W = 0  \]
in the weak sense. Since $g$ has no Killing fields this implies $W = 0$ and hence $U = 0$, which is impossible since we assumed $(\pi,U)$ non everywhere zero. 
Now by \eqref{expressions} it is easily seen that 
\begin{equation} \label{ineqexistence}
C(n,g,V,\psi) \le C(n)  
 V_g^{-1} \left( S_{\mathcal{R}_\psi} \max_{x\in M} \mathcal{B}_{\tau, \psi,V}(x)  \right)^{1-n} \left( \int_M \mathcal{R}_\psi dv_g \right)^{- \frac{2^*}{2}}
\end{equation}
where we used the notations of \eqref{expressions} and where $C(n,g,V,\psi)$ is as in \eqref{C(n,M)}. 
Now we consider the equation
\begin{equation} \label{ELeta}
\triangle_g \vp + \mathcal{R}_\psi \vp = \mathcal{B}_{\tau, \psi, V} \vp^{2^*-1} + \mathcal{A}_{\pi,U} \left( W(\eta) \right) \vp^{-2^*-1}~.
\end{equation}
We claim that \eqref{ELeta} has a smooth positive solution for any $\eta \in B_{N_m}$. We just proved that $\mathcal{A}_{\pi,U}(W(\eta))$ is never zero. 
Thus we can construct subsolutions of \eqref{ELeta} as small as we want as we did in Section \ref{solumini}, see \eqref{soussol}. 
On the other hand, by \eqref{estimeeLgW} and \eqref{ineqexistence}, there holds that for any $\eta \in B_{N_m}$ and any $\delta > 0$ sufficiently small,
\[ \mathcal{A}_{\pi,U}(W(\eta)) + \delta \le C(n) 
 V_g^{-1} \left( S_{\mathcal{R}_\psi} \max_{x\in M} \mathcal{B}_{\tau, \psi,V}(x)  \right)^{1-n} \left( \int_M \mathcal{R}_\psi dv_g \right)^{- \frac{2^*}{2}}
~. \]
Then the existence result by Hebey-Pacard-Pollack, namely Theorem $3.1$ and equation $(3.3)$ in \cite{HePaPo}, applies to \eqref{ELeta} when replacing $\mathcal{A}_{\pi,U}(W(\eta))$ by $\mathcal{A}_{\pi,U}(W(\eta)) + \delta$ and provides us with a strict super solution of \eqref{ELeta}.  Since $\mathcal{A}_{\pi,U}(W(\eta))$ is nonzero for all $\eta \in B_{N_m}$, and since it is smooth, Proposition \ref{solmin} shows that \eqref{ELeta} 
possesses a minimal smooth positive solution $\vp(\mathcal{A}_{\pi,U}(W(\eta)))$, where we use the same notations as in Proposition \ref{solmin}. The following map:  
\begin{equation} \label{defT}
\mathcal{T}:  \eta \mapsto \mathcal{T}(\eta) = \vp \left( \mathcal{A}_{\pi,U}(W(\eta)) \right)
\end{equation} 
is thus well-defined in $B_{N_m} $. It is clear that a fixed point of $\mathcal{T}$ is a solution of the constraint system. 
As a consequence of the monotonicity property of the minimal solution in Proposition \ref{solmin} along with \eqref{estimeeLgW} and the very definition of $N_m$ in \eqref{defN1} we obtain that, for any $\eta \in B_{N_m}$, 
\begin{equation} \label{boulestable}
0 <  \mathcal{T}(\eta) \le N_m.
\end{equation}
Hence $B_{N_m}$ is stable under $\mathcal{T}$ and $\mathcal{T}$ maps $B_{N_m}$ into itself. Now we prove that $\mathcal{T}$ is continuous in $B_{N_m}$.

\paragraph{$\cal{T}$ is continuous in $B_{N_m}$.} \label{continuiteT}

First we  claim that there exists a positive real number $\delta_0$ such that for any $\eta \in B_{N_m}$, and any $x \in M$,
\begin{equation} \label{infunif}
\mathcal{T}(\eta)(x) \ge \delta_0.
\end{equation}
To prove this claim we pick a sequence $(\eta_k)_k$ in $B_{N_m}$ and show that there holds
\begin{equation} \label{infunif2}
\liminf_{k \to + \infty} \min_{x \in M} \mathcal{T}(\eta_k)(x) > 0. 
\end{equation}
We consider the associated $W(\eta_k)$ as in \eqref{Weta}. By Proposition \ref{estimee}, if we choose $\ve(n,g,V,\psi)$ in (4.6) such that 
\[ \ve(n,g,V,\psi) < \frac{N_m^{2^\star}}{\frac{n-1}{n}N_m^{2^\star}+\Vert\nabla\psi\Vert_\infty}~,\]
then there holds 
\begin{equation} \label{continuiteW}
\Vert W(\eta_k) \Vert_{C^{1,\alpha}} \le  C_0 N_m^{2^*} 
\end{equation}
so that, up to a subsequence, we can assume that $W(\eta_k)$ converges to some $W_0$ in the $C^{1,\alpha}(M)$-topology for some $0 < \alpha < 1$. As noticed in a previous remark right after defining $\mathcal{T}$,  $\mathcal{A}_{\pi,U}(W_0)$ is non-negative and positive somewhere in $M$. We denote by $x_0$ its maximum point and choose $0< r < i_g(M)$ so as to have 
\[ \mathcal{A}_{\pi,U}(W_0)(x) \ge \frac{1}{2} \mathcal{A}_{\pi,U}(W_0)(x_0)\] 
for all $x \in B_{x_0}(2 r)$. Let $\lambda$ be a smooth nonnegative function, compactly supported in $B_{x_0}(2 r)$ and equal to $1$ in  $B_{x_0}(r)$.  Since $W(\eta_k) \xrightarrow{C^{1,\alpha}} W_0$ as $k$ goes to infinity one has,
for $k$ large enough, in $M$, that
\[ \mathcal{A}_{\pi,U}(W(\eta_k)) \ge \frac{1}{2} \lambda \mathcal{A}_{\pi,U}(W_0). \]
The monotonicity property in Proposition \ref{solmin} and the definition of $\mathcal{T}$ in \eqref{defT} thus show that
\begin{equation*} 
\mathcal{T}(\eta_k) \ge \vp \left( \frac{1}{2} \lambda \mathcal{A}_{\pi,U}(W_0) \right)~,
\end{equation*} 
where the right-hand side is a smooth positive function in $M$, which shows \eqref{infunif2}. 

\medskip
Now we prove the continuity of $\mathcal{T}$. We let $\eta_k \in B_{N_m}$ be a sequence of nonnegative functions in $M$ converging uniformly to some $\eta_0 \in B_{N_m}$. There holds that 
\begin{equation*}
\triangle_g \mathcal{T}(\eta_k) + \mathcal{R}_\psi\mathcal{T}(\eta_k) = \mathcal{B}_{\tau, \psi, V} \mathcal{T}(\eta_k)^{2^*-1} 
+ \mathcal{A}_{\pi,U} \left( W(\eta_k) \right)\mathcal{T}(\eta_k)^{-2^*-1}
\end{equation*}
for all $k$. 
By \eqref{boulestable} and \eqref{infunif}, standard elliptic theory shows that $\mathcal{T}(\eta_k)$ converges in the $C^2(M)$-topology, up to a subsequence, to some $\mathcal{T}_0$ solution of
\begin{equation} \label{eqlim}
 \triangle_g \mathcal{T}_0 + \mathcal{R}_\psi \mathcal{T}_0  = \mathcal{B}_{\tau,\psi,V} \mathcal{T}_0^{2^*-1} + \mathcal{A}_{\pi,U}(W(\eta_0)) \mathcal{T}_0^{-2^*-1}. 
 \end{equation}
There holds then by Proposition \ref{solmin}: either  $\mathcal{T}_0 = \mathcal{T}(\eta_0)$ or $\mathcal{T}_0 >  \mathcal{T}(\eta_0)$ everywhere. We proceed by contradiction and assume that $ \mathcal{T}_0 > \mathcal{T}(\eta_0)$. We then define for any $t \in [0;1]$
\[ m(t) =  I_0 \big( t\mathcal{T}(\eta_0) + (1-t) \mathcal{T}_0 \big)~, \]
where $I_0$ is defined for any positive $\eta \in H^1(M)$ and is the energy associated to \eqref{eqlim}:
\begin{equation} \label{iw}
I_0(\eta) = \frac{1}{2}\int_M \left( |Ê\nabla \eta |^2 + \mathcal{R}_\psi \eta^2 \right) dv_g - \frac{1}{2^*} \int_M \mathcal{B}_{\tau, \psi,V} \eta^{2^*} dv_g + \frac{1}{2^*} \int_M \mathcal{A}_{\pi,U}(W(\eta_0)) \eta^{-2^*} dv_g .
\end{equation}
By proposition \ref{solmin} each $\mathcal{T}(\eta_k)$ is a stable solution, and thus $\mathcal{T}_0$ is stable. Hence $m''(0) \ge 0$. Using \eqref{iw} we can compute $m^{(3)}(t)$ for any $t \in [0,1]$, where 
$m^{(3)}$ is the third derivative of $m$. There holds
\[ \begin{aligned}
m^{(3)}(t)   = -(2^*-1)(2^*-2) \int_M \mathcal{B}_{\tau, \psi,V} \big(t \mathcal{T}(\eta_0) + (1-t) \mathcal{T}_0 \big)^{2^*-3}(\mathcal{T}(\eta_0) - \mathcal{T}_0)^3 dv_g \\
- (2^*+1)(2^*+2) \int_M \mathcal{A}_{\pi,U}(W(\eta_0)) \big( t \mathcal{T}(\eta_0) + (1-t) \mathcal{T}_0 \big)^{-2^*-3} (\mathcal{T}(\eta_0) - \mathcal{T}_0)^3 dv_g .\\
\end{aligned} \]
Since $\mathcal{T}_0 > \mathcal{T}(\eta_0)$ and $\mathcal{B}_{\tau,\psi,V}$ is nonnegative nonzero, $m^{(3)}(t)$  is positive for all $t \in (0,1)$. Hence $m''$ is a positive function of $t$ for $0 < t \le 1$ and $m'$ is increasing in $(0,1)$. But this is impossible since both $\mathcal{T}_0$ and $\mathcal{T}(\eta_0)$ are solutions of \eqref{eqlim} and there thus holds $m'(0) = m'(1) = 0$. Hence $\mathcal{T}_0 = \mathcal{T}(\eta_0)$ and in particular $\mathcal{T}$ is continuous. In order to apply the Schauder's fixed point theorem, it remains to prove the precompactness of 
$\mathcal{T}(B_{N_m})$, which itself follows from the compactness of $\mathcal{T}$.

\paragraph{Compactness of $\mathcal{T}$ and conclusion.}

Clearly, $B_{N_m}$ is a closed convex set in $L^\infty_+(M)$. It remains to show that $\mathcal{T}(B_{N_m})$ is compact to conclude. By \eqref{estimeeLgW}, \eqref{boulestable} and  \eqref{infunif}, for any $\eta \in B_{N_m}$, $\mathcal{T}(\eta)$ satisfies $\delta_0 \le \mathcal{T}(\eta) \le N_m$ and:
\[ \triangle_g \mathcal{T}(\eta) + \mathcal{R}_\psi \mathcal{T}(\eta) \le \Vert \mathcal{B}_{\tau, \psi,V} \Vert_\infty N_m^{2^*-1} + \delta_0^{- 2^*-1} C(n,g,V,\psi)~, \]
where $C(n,g,V,\psi)$ is as in \eqref{C(n,M)}. By standard elliptic theory $\mathcal{T}(B_{N_m})$ is thus bounded in $C^1(M)$. By the compactness of the embedding $C^1(M) \subset L^\infty(M)$
we then get that $\mathcal{T}(B_{N_m})$ is a compact set of $L^\infty_+(M)$. Applying Schauder's fixed point theorem yields the existence of a fixed-point of $\mathcal{T}$ on $B_{N_m}$, i.e. a solution of the constraint system, and concludes the proof of Theorem \ref{Th}.

\normalsize

\bibliographystyle{abbrv}
\bibliography{biblio}

\bigskip

\scriptsize
\noindent BRUNO PREMOSELLI,
D\'EPARTEMENT DE MATH\'EMATIQUES, UNIVERSIT\'E DE CERGY-PONTOISE, SITE DE SAINT-MARTIN, $2$ AVENUE ADOLPHE CHAUVIN, $95302$, CERGY-PONTOISE CEDEX, FRANCE AND
D\'EPARTEMENT DE MATH\'EMATIQUES - UMPA, ECOLE NORMALE SUP\'ERIEURE DE LYON, $46$ ALLEE D'ITALIE, $69364$ LYON CEDEX $07$, FRANCE.

\medskip
\textit{E-mail address}: bruno.premoselli@u-cergy.fr

\end{document}